\documentclass{amsart}
\usepackage{graphicx}
\usepackage{amssymb}

\newtheorem{thm}{Theorem}[section]

\newtheorem{lem}[thm]{Lemma}

\theoremstyle{definition}

\theoremstyle{remark}
\newtheorem{rem}[thm]{Remark}

\begin{document}

\title[Krylov-Safonov without localization]{A proof of the Krylov-Safonov theorem without localization}
\author{Connor Mooney}
\address{Department of Mathematics, UC Irvine}
\email{\tt mooneycr@math.uci.edu}

\begin{abstract}
The Krylov-Safonov theorem says that solutions to non-divergence uniformly elliptic equations with rough coefficients are H\"{o}lder continuous. The proof combines a basic measure estimate with delicate localization and covering arguments. Here we give a ``global" proof based on convex analysis that avoids the localization and covering arguments. As an application of the technique we prove a $W^{2,\,\epsilon}$ estimate where $\epsilon$ decays with the ellipticity ratio of the coefficients at a rate that improves previous results, and is optimal in two dimensions.
\end{abstract}
\maketitle

\section{Introduction}
In this paper we consider viscosity solutions of
\begin{equation}\label{Main}
M_{\Lambda}^-(D^2u) \leq 0 \leq M_{\Lambda}^+(D^2u)
\end{equation}
in $B_1 \subset \mathbb{R}^n$, where $\Lambda \geq 1$ and $M_{\Lambda}^{\pm}$ are the Pucci extremal operators (see Section \ref{Statements} for the definition). The problem (\ref{Main}) includes all $C^2$ solutions to
uniformly elliptic equations of the form $\text{tr}(A(x)D^2u) = 0$ with $I \leq A(x) \leq \Lambda I$. 

Fundamental results for (\ref{Main}) include the interior $C^{\alpha}$ estimate of Krylov-Safonov \cite{KS}, and Lin's interior $W^{2,\,\epsilon}$ estimate \cite{Lin}. Here
$\alpha$ and $\epsilon$ depend only on $n$ and $\Lambda$. These results have important
consequences for the regularity of solutions to fully nonlinear uniformly elliptic equations of the form $F(D^2w) = 0$, including: interior $C^{1,\,\alpha}$ estimates for general $F$ (see \cite{CC}); 
the Evans-Krylov interior $C^{2,\,\alpha}$ estimates for concave $F$ (\cite{E}, \cite{K}); and partial regularity for general $F$ (\cite{ASS}).

To our knowledge, all the proofs of the $C^{\alpha}$ and $W^{2,\,\epsilon}$ estimates for (\ref{Main}) combine a basic (ABP-type) measure estimate with a localization (barrier) argument and a delicate covering (Calderon-Zygmund
or Vitali) argument. The purpose of this paper is to give ``global" proofs of these results that completely avoid the localization and covering arguments, and instead rely on elementary convex analysis.

Apart from its own interest, our method is useful for understanding the dependence of $\epsilon$ on $\Lambda$ in $W^{2,\,\epsilon}$ estimates for supersolutions of (\ref{Main}).
This question received attention recently due to its connection to partial regularity for fully nonlinear equations.
In \cite{ASS}, the authors conjecture that $\epsilon$ depends linearly on $\Lambda^{-1}$ (in any dimension $n$), and they construct an example showing this is the best we can hope for.
In \cite{Le}, Le proves that we can take $\epsilon \geq c(n)\Lambda^{-1-n}$, improving the exponential dependence previously known. With our method we improve further to
$$\epsilon \geq c(n)\Lambda^{1-n}.$$ In view of known examples, this dependence is optimal when $n = 2$. 

\vspace{3mm}

We now describe our approach. The idea is to show decay in measure of the sets where $u$ lies above its lower envelope of paraboloids with opening $-2^{k}$ (see Section \ref{Preliminaries} for the definition) by taking the paraboloids of opening $-2^{k+1}$ which are tangent to the envelope away from the agreement set, and sliding them up until they touch $u$. Our key lemma (Lemma \ref{MeasureEstimate}) is a basic measure estimate which says that provided the new contact points are interior points, they fill a universal fraction of the set where $u$ lies above its envelope. The main issue is thus to make sure that the new contact points are interior points.

Using this approach we prove three results for super-solutions of (\ref{Main}): an interior $W^{2,\epsilon}$ estimate, a global $W^{2,\epsilon}$ estimate, and the weak Harnack inequality (which by standard arguments implies the Krylov-Safonov interior H\"{o}lder estimate for solutions of (\ref{Main})).
For the interior $W^{2,\,\epsilon}$ estimate we can reduce to the case that $u$ agrees with a paraboloid near the boundary, which guarantees interior contact points.
For the global $W^{2,\epsilon}$ estimate, we must instead lift the paraboloids which are tangent to the envelope at least distance $\sim 2^{-k/2}$ from the boundary to get interior contact points.
We then use a dichotomy argument (either we can proceed as in the interior estimate, or the set where $u$ lies above its envelope concentrates in a set of small measure near the boundary). Our proof of the weak Harnack inequality follows the same lines as that of the global $W^{2,\,\epsilon}$ estimate. 

\vspace{3mm}

The paper is organized as follows. In Section \ref{Statements} we give precise statements of our results.
In Section \ref{Preliminaries} we recall some notions from convex analysis, and prove our version of the basic measure estimate. In Section \ref{InteriorW2E} we prove the interior $W^{2,\,\epsilon}$ estimate,
with $\epsilon \sim \Lambda^{1-n}$. In Section \ref{GlobalW2E} we prove the global $W^{2,\,\epsilon}$ estimate.
Finally, in Section \ref{WeakHarnack} we prove the weak Harnack inequality.

\section{Statements of Results}\label{Statements}

In this section we give precise statements of our results. To that end we make some definitions. For $N \in \text{Sym}_{n \times n}$ we define the Pucci extremal operators
$$M_{\Lambda}^-(N) := \left(\sum \text{positive eigenvalues of N}\right) + \Lambda\,\left(\sum \text{negative eigenvalues of N}\right),$$
and 
$$M_{\Lambda}^+(N) := -M_{\Lambda}^-(-N).$$
It is straightforward to check that $M_{\Lambda}^-(N) \leq \text{tr}(AN) \leq M_{\Lambda}^+(N)$ for all $I \leq A \leq \Lambda I$, and each
inequality is achieved for some choice of admissible $A$ depending on $N$. For a bounded domain $\Omega \subset \mathbb{R}^n$ and
functions $u,\, \varphi \in C(\overline{\Omega})$ we say that $\varphi$ is tangent from below to $u$ in $\Omega$ at $x_0 \in \Omega$ if $\varphi \leq u$ in $\overline{\Omega}$
and $\varphi(x_0) = u(x_0)$. Finally, we say that $M_{\Lambda}^-(D^2u) \leq 0$ in the viscosity sense in $\Omega$ if $M_{\Lambda}^-(D^2\varphi(x_0)) \leq 0$ whenever
$\varphi \in C^2(\overline{\Omega})$ and $\varphi$ is tangent from below to $u$ in $\Omega$ at $x_0 \in \Omega$.

\vspace{3mm}

Now we recall an important second-order quantity.
We say that a paraboloid $P$ has opening $a$ if $D^2P = aI$. For a bounded domain $\Omega \subset \mathbb{R}^n$ and a function
$v \in C(\overline{\Omega})$, we define for $x \in \Omega$ the function $\underline{\Theta}_v(x)$ to be smallest $a \geq 0$ such that
a paraboloid of opening $-a$ is tangent from below to $v$ in $\Omega$ at $x$. If there is no tangent paraboloid from below at $x$
(e.g. at the origin for $v = -|x|$ in $\Omega = B_1$), then we say $\underline{\Theta}_v(x) = \infty$.

\vspace{3mm}

We now state our results. They are all estimates for super-solutions of (\ref{Main}). Our first result is an interior $L^{\epsilon}$ estimate for $\underline{\Theta}$, 
with quantitative dependence of $\epsilon$ on $\Lambda$:
\begin{thm}\label{LocalW2E}
Assume that $u \in C(\overline{B_1})$, with $M^{-}_{\Lambda}(D^2u) \leq 0$ in $B_1$. Then
$$|\{\underline{\Theta}_u > 64\|u\|_{L^{\infty}(B_1)}\,t\} \cap B_{1/2}| \leq |B_2|\, t^{-\epsilon}$$
for all $t \geq 2$, with
$$\epsilon = 2^{-2}\,(2n\Lambda)^{1-n}.$$
\end{thm}
\noindent It is standard that as a consequence of Theorem (\ref{LocalW2E}), we obtain the interior $W^{2,\,\delta}$ estimate
$$\|D^2u\|_{L^{\delta}(B_{1/2})} \leq C(n,\,\Lambda)\|u\|_{L^{\infty}(B_1)}$$
for solutions of (\ref{Main}), with $\delta = \epsilon/2 = 2^{-3}(2n\Lambda)^{1-n}$.

Estimates of this type were previously proven in \cite{ASS} (for some small $\epsilon(n,\,\Lambda)$) and more recently in \cite{Le} (with $\epsilon \sim \Lambda^{-1-n}$).
Theorem \ref{LocalW2E} improves the dependence further to $\epsilon \sim \Lambda^{1-n}$, which is optimal when $n = 2$ in view of the examples in \cite{ASS}. 
See Remark \ref{Gain} for an explanation of exactly where the two extra factors of $\Lambda$ are gained in our argument.

\begin{rem}
We conjecture that the dependence $\epsilon \sim \Lambda^{1-n}$ is optimal in any dimension. For $x \in \mathbb{R}^n$ write $x = (x',\,x_n)$. 
Preliminary constructions using building blocks of the form $min\{\Lambda |x'|^2 - x_n^2 - 1,\, 0\}$, along with its
rescalings and rotations, suggest the existence of examples where $\epsilon \sim \Lambda^{\frac{1-n}{2}}$. This would show at least that the decay rate depends on dimension. We intend to investigate this in future work.
\end{rem}

\begin{rem}
In \cite{ASS} and \cite{Le} the key estimate is stated slightly differently, namely: 
\begin{equation}\label{Correction}
|\{\underline{\Theta}_u > t\} \cap B_{1/2}| < C(n,\,\Lambda)t^{-\epsilon}
\end{equation}
for all $t > t_0(n,\,\Lambda)\|u\|_{L^{\infty}(B_1)}$. However, the authors actually proved (and we believe meant to state) an estimate of the form we write
in Theorem \ref{LocalW2E}. Indeed, consider $u = s(1-|x|)$ and $t = 2t_0s$. By the one-homogeneity of $\underline{\Theta}_u$ in $u$,
as $s \rightarrow \infty$ the left side of (\ref{Correction}) is constant and positive while the right side goes to $0$. 
\end{rem}

Our second theorem is a global $L^{\epsilon}$ estimate for $\underline{\Theta}$:
\begin{thm}\label{GlobW2E}
Under the same hypotheses as Theorem \ref{LocalW2E} we have
$$|\{\underline{\Theta}_u > 2^{10n}\|u\|_{L^{\infty}(B_1)}\,t\}| \leq |B_1|\, t^{-\epsilon}$$
for all $t \geq 2$, with $\epsilon = 2^{-3}\,(2n\Lambda)^{1-n}.$
\end{thm}
\noindent Again, a global $W^{2,\,\delta}$ estimate for solutions of (\ref{Main}) follows, with $\delta = \epsilon/2.$ Global $W^{2,\,\delta}$ estimates 
for (\ref{Main}) for some $\delta(n,\,\Lambda)$ were proven in \cite{LL}.
Our main reason for including Theorem \ref{GlobW2E} (apart from the improved dependence
of $\delta$ on $\Lambda$ compared to previous results) is to emphasize the method of proof, which avoids using localizing barriers
and covering arguments, and involves a dichotomy argument that we also use in our proof of the weak Harnack inequality.

\vspace{3mm}

Our last result is an $L^{\epsilon}$ estimate for positive super-solutions known as the ``weak Harnack inequality:"
\begin{thm}\label{LEps}
Assume that $M_{\Lambda}^-(D^2u) \leq 0$ in $B_4 \subset \mathbb{R}^n$, and that $u \geq 0$. Then
$$|\{u \geq 2^{64n^2\Lambda}u(0) t\} \cap B_{1/2}| \leq |B_{1/2}|t^{-\epsilon}$$
for $t \geq 2$, with $\epsilon = 2^{-3}(2n\Lambda)^{1-n}$.
\end{thm}
\noindent It is standard to obtain from Theorem \ref{LEps} the Krylov-Safonov interior $C^{\alpha}$ estimate
$$\|u\|_{C^{\alpha}(B_{1/2})} \leq C(n,\,\Lambda)\|u\|_{L^{\infty}(B_1)}$$
for solutions of (\ref{Main}) in $B_1$ (see e.g. \cite{CC}), and Theorem \ref{LEps} gives $\alpha \sim c(n)^{\Lambda^{n-1}}$.

\section{Preliminaries}\label{Preliminaries}

In this section we recall some notions from from convex analysis, and we prove a basic measure estimate.

\vspace{3mm}

For $y \in \mathbb{R}^n$ and $a,\,b \in \mathbb{R}$ we define the paraboloids $P^a_{y,\,b}$ by
$$P^a_{y,\,b}(x) := \frac{a}{2}|x|^2 + y \cdot x + b.$$ 
We say the paraboloids $P^a_{y,\,b}$ have opening $a$. For a bounded, strictly convex domain $\Omega \subset \mathbb{R}^n$ and a function
$v \in C(\overline{\Omega})$, we define the $a$-convex envelope $\Gamma_v^a$ on $\overline{\Omega}$ by
$$\Gamma_v^a(x) := \sup_{y \in \mathbb{R}^n,\, b \in \mathbb{R}} \{P^{-a}_{y,\,b}(x) : P^{-a}_{y,\,b} \leq v \text{ in } \overline{\Omega}\}.$$
When $a = 0$, $\Gamma^a_v$ is the usual convex envelope. Using the strict convexity of $\Omega$ it is not hard to show that $\Gamma_v^a = v$ on $\partial \Omega$.
Finally, we define
$$A_a(v) := \{x \in \Omega: v(x) = \Gamma^a_v(x)\}$$
to be those points in $\Omega$ where $v$ has a tangent paraboloid of opening $-a$ from below in $\Omega$. It is straightforward to show that
$A_a(v)$ is closed in $\Omega$, that $A_a(v) \subset A_{\tilde{a}}(v)$ when $a \leq {\tilde{a}}$, and that
\begin{equation}\label{aconvex}
\Gamma_{\mu v + \frac{\gamma}{2}|x|^2}^{\lambda} = \mu\, \Gamma_v^{\frac{\lambda + \gamma}{\mu}} + \frac{\gamma}{2}|x|^2 \quad \text{and} \quad
A_{\lambda}\left(\mu v + \frac{\gamma}{2}|x|^2\right) = A_{\frac{\lambda + \gamma}{\mu}}(v)
\end{equation}
for any $\lambda, \gamma \in \mathbb{R}$ and $\mu > 0$.

\vspace{3mm}

The key result from this section is a measure estimate (Lemma \ref{MeasureEstimate}) for super-solutions of (\ref{Main}). 
Here and below, $\Omega$ is a bounded, strictly convex domain in $\mathbb{R}^n$.
We begin with some simple observations.
\begin{lem}\label{NewContact}
Assume that $v \in C(\overline{\Omega})$ satisfies $M_{\Lambda}^-(D^2v) \leq K < \infty$. Assume that a paraboloid $P$ of opening $-a < 0$ is tangent from below to $\Gamma^0_v$ 
in $\Omega$ at a point $x_0 \in \Omega \backslash A_0(v)$, and slide $P$ up until $P + t$ touches $v$ from below at $x_1 \in \overline{\Omega}$ for some $t > 0$. 
Then $x_1 \in \overline{\Omega} \backslash A_0(v)$.
\end{lem}
\begin{proof}
Assume by way of contradiction that $x_1 \in A_0(v)$. Let $L$ a supporting affine function to $\Gamma^0_v$ at $x_1$.
If $P+t$ is tangent from below to $L$ at $x_1$ then we conclude that $L$ (hence $\Gamma^0_v$) is strictly larger than $P$, a contradiction.
We may thus assume after subtracting $L$, translating, rotating, rescaling, and multiplying by a constant, that $\max\{0,\, 1-|x|^2\}$ is tangent from below to $v$
at $e_1 \in \Omega$. 

It follows that for all $A > 0$ the function $\varphi_A := 1- |x| + A(|x|-1)^2$ is tangent from below to $v$ in a neighborhood 
(depending on $A$) of $e_1$. We compute $M_{\Lambda}^-(D^2\varphi_A(e_1)) = 2A - (n-1)\Lambda$. For $A$ sufficiently
large we contradict that $M_{\Lambda}^-(D^2v)$ is bounded.
\end{proof}

\begin{lem}\label{ExpandingMeasure}
Assume that $v \in C(\overline{\Omega})$ is convex. For any measurable set $F \subset \Omega$, let $V$ denote the set of vertices of all tangent paraboloids of opening 
$-a < 0$ to $v$ at points in $F$. Then $|V| \geq |F|.$
\end{lem}
\begin{proof}
The set $V$ consists of those points of the form $x + \frac{1}{a}p$ where $x \in F$ and $p \in \partial v(x)$. Here $\partial$ denotes sub-gradient.
We conclude that $V = \partial w(F),$ where $w := \frac{1}{2}|x|^2 + \frac{1}{a}v$.
Since $\det D^2w \geq \det D^2(|x|^2/2) = 1$ in the Alexandrov sense (see e.g. \cite{Gut} for the definition), the result follows.
\end{proof}

Finally, we prove the measure estimate.
\begin{lem}\label{MeasureEstimate}
Assume that $u \in C(\overline{\Omega})$, and $M_{\Lambda}^-(D^2u) \leq 0$ in $\Omega$. For $a > 0$ and a measurable set $F \subset \{u > \Gamma^a_u\}$, 
take the tangent paraboloids of opening $-2a$ to $\Gamma^a_u$ on $F$, and slide them
up until they touch $u$ on a set $E$. If $E \subset\subset \Omega$, then
$$|A_{2a}(u) \backslash A_a(u)| \geq 2^{-1}(2n\Lambda)^{1-n}|F|.$$
\end{lem}
\begin{proof}
By inner approximation with closed sets, it suffices to prove the Lemma when $F$ is closed.
Let $v = \frac{1}{a}u + \frac{1}{2}|x|^2$. Then using (\ref{aconvex}) we have that $A_a(u) = A_0(v), \, A_{2a}(u) = A_1(v),$ and that
the set $E$ is obtained by taking the paraboloids of opening $-1$ which are tangent from below to $\Gamma^0_v$ on $F$, and sliding them up until they touch $v$.
Let $V$ denote the set of vertices of these paraboloids. We remark that $V$ and $E$ are closed (in particular, measurable).

\vspace{3mm}

We first claim that $E \subset A_{2a}(u) \backslash A_{a}(u)$. Indeed, we have
\begin{equation}\label{Equation}
M_{\Lambda}^-(D^2v) \leq \frac{1}{a}M_{\Lambda}^-(D^2u) + M_{\Lambda}^+(I) \leq n\Lambda
\end{equation}
in the viscosity sense. By Lemma \ref{NewContact} we have $E \cap A_0(v) = \emptyset$, and by hypothesis $E \cap \partial \Omega = \emptyset$. Since
$E \subset A_1(v)$ by definition, we conclude that $E \subset A_1(v) \backslash A_0(v) = A_{2a}(u) \backslash A_{a}(u)$.

\vspace{3mm}

Assume for the moment that $u$ is semi-concave (that is, for some $M > 0$, $u$ admits a tangent paraboloid of opening $M$ from above in $\Omega$ at every point). 
By Alexandrov's theorem, $v$ is almost everywhere twice differentiable (see e.g. \cite{EG}), and at a point $x_0 \in \Omega$ of twice differentiability we have by (\ref{Equation}) and a standard argument that
\begin{equation}\label{PointwiseEquation}
M_{\Lambda}^-(D^2v(x_0)) \leq n\Lambda.
\end{equation}
Let $x \in E$ with corresponding vertex $y \in V$. By semi-concavity, $v$ is differentiable at $x$ and the vertex $y$ is given by
$$y = x + \nabla v(x).$$
Furthermore, the map $x \rightarrow y$ is Lipschitz on $E$. If in addition $u$ is twice differentiable at $x$, then $D^2v(x) \geq -I$ (since $x \in A_1(v)$). 
Combining this with (\ref{PointwiseEquation}) we obtain $D^2v(x) \leq (2n-1)\Lambda I$. Note also that the smallest eigenvalue of $D^2u(x)$ is nonpositive,
since $M_{\Lambda}^-(D^2u(x)) \leq 0$. It follows that the smallest eigenvalue of $D^2v(x)$ is at most $1$. We conclude that
$$\det D_x y = \det (I + D^2v) \leq 2(2n\Lambda)^{n-1}$$
almost everywhere on $E$. Integrating over $E$ and using the area formula we obtain
$|V| \leq 2(2n\Lambda)^{n-1}|E|.$
Finally, by Lemma \ref{ExpandingMeasure} we have $|F| \leq |V|$. Using that $E \subset A_{2a}(u) \backslash A_{a}(u)$, we obtain the desired estimate when $u$ is semi-concave.

\vspace{3mm}

It is standard that we can reduce to the case that $u$ is semi-concave by using inf-convolutions; see e.g. the proof of Lemma $2.1$ in \cite{S}. 
For the reader's convenience we sketch the argument. Using inf-convolution we can find for any $\Omega'$ with $E \subset\subset \Omega' \subset\subset \Omega$ 
a sequence of semi-concave super-solutions $u_k$ of (\ref{Main}) in $\Omega'$ that converge uniformly to $u$. Let $v_k := \frac{1}{a}u_k + \frac{1}{2}|x|^2$, and slide the paraboloids of opening $-1$ with vertices in $V$ from below $v_k$ until they touch
on a set $E_k$. For $k$ large we have $E_k \subset \subset \Omega'$, and the above argument gives $|E_k| \geq 2^{-1}(2n\Lambda)^{1-n}|V|$. It is straightforward
to show that 
$$\lim\sup_{k \rightarrow \infty} E_k = \cap_{k \geq 1} \cup_{m \geq k} E_m \subset E,$$
and the conclusion follows as above.
\end{proof}

\section{Interior $W^{2,\,\epsilon}$ Estimate}\label{InteriorW2E}

In this section we prove the interior $W^{2,\,\epsilon}$ estimate Theorem \ref{LocalW2E}. We reduce to the case that $u$ agrees with a paraboloid near the boundary,
which makes the argument simple.

\begin{proof}[{\bf Proof of Theorem \ref{LocalW2E}:}]
After dividing by $2^6\|u\|_{L^{\infty}(B_1)}$ and adding a constant we may assume that $0 < u \leq 2^{-4}$ in $\overline{B_1}$. Define the extension
$$\tilde{u} = \begin{cases}
\min\left\{u,\, \frac{1}{4}(1-|x|^2)\right\}, \quad x \in B_1 \\
\frac{1}{4}(1-|x|^2), \quad x \in B_R \backslash B_1
\end{cases}$$
for $R$ large (say $32$). It is straightforward to check that $\tilde{u} \in C(\overline{B_R})$ with $\tilde{u} = u$ in $B_{3/4}$, and that
$B_R \backslash A_1(\tilde{u}) \subset \subset B_2$. Furthermore, by basic properties of super-solutions we have that $M^{-}_{\Lambda}(D^2\tilde{u}) \leq 0$ in $B_R$.

We will prove by induction that
$$|B_2 \backslash A_{2^k}(\tilde{u})| \leq (1-2^{-1}\,(2n\Lambda)^{1-n})^k|B_2|$$
for all $k \geq 0$. The case $k = 0$ is obvious, so we proceed to the inductive step. 
Let $F_k = B_R \backslash A_{2^k}(\tilde{u}) = B_2 \backslash A_{2^k}(\tilde{u})$.
Take the paraboloids of opening $-2^{k+1}$ which are tangent from below to $\Gamma^{2^k}_{\tilde{u}}$ on $F_k$, and slide them up until they
touch $\tilde{u}$ on a set $E_{k}$. Since $R$ is large and $F_k \subset B_2$ it is easy to see from the
definition of $\tilde{u}$ that $E_{k} \subset\subset B_R$. 
We conclude using Lemma \ref{MeasureEstimate} that
\begin{equation}\label{InteriorMeasureBump}
2^{-1}(2n\Lambda)^{1-n}|B_2 \backslash A_{2^k}(\tilde{u})| = 2^{-1}(2n\Lambda)^{1-n}|F_k| \leq |A_{2^{k+1}}(\tilde{u}) \backslash A_{2^k}(\tilde{u})|.
\end{equation}
Since $|B_2 \backslash A_{2^{k+1}}(\tilde{u})| = |B_2 \backslash A_{2^k}(\tilde{u})| - |A_{2^{k+1}}(\tilde{u}) \backslash A_{2^k}(\tilde{u})|$ it follows from (\ref{InteriorMeasureBump}) that
$$|B_2 \backslash A_{2^{k+1}}(\tilde{u})| \leq (1-2^{-1}\,(2n\Lambda)^{1-n})|B_2 \backslash A_{2^k}(\tilde{u})|,$$
which completes the inductive step. 

Since $\tilde{u} \leq u$ and they agree in $B_{3/4}$, we have for $k \geq 0$ that
$$\{\underline{\Theta}_u > 2^k\} \cap B_{1/2} \subset B_2 \backslash A_{2^k}(\tilde{u}).$$ 
We conclude that
$$|\{\underline{\Theta}_u > 2^k\} \cap B_{1/2}| \leq (1-2^{-1}(2n\Lambda)^{1-n})^k|B_2|,$$
and Theorem \ref{LocalW2E} follows quickly. More precisely, for $2^k \leq t < 2^{k+1}$ we get the desired inequality provided
$$\epsilon \leq -\frac{k}{(k+1)\log(2)}\log(1-2^{-1}(2n\Lambda)^{1-n})$$
which is guaranteed for all $k \geq 1$ when $\epsilon = 2^{-2}(2n\Lambda)^{1-n}$ by elementary calculus.
\end{proof}

\begin{rem}\label{Gain}
There are two places where we gain a factor of $\Lambda$ with respect to the main result of \cite{Le}. First, in our basic measure estimate Lemma \ref{MeasureEstimate}
we use that at least one eigenvalue of $D^2u$ is nonpositive; and second, our method avoids the use of localizing barriers, which increase the paraboloid opening by a factor $\sim 2^{\Lambda}$
rather than $2$ at each stage of the iteration in \cite{Le}.
\end{rem}

\section{Global $W^{2,\,\epsilon}$ Estimate}\label{GlobalW2E}

In this section we prove the global $W^{2,\,\epsilon}$ estimate Theorem \ref{GlobW2E}. To apply Lemma \ref{MeasureEstimate} we slide paraboloids that are tangent in balls $B_{1-\rho_k}$ that expand quickly to fill $B_1$,
and use a dichotomy argument.

\begin{proof}[{\bf Proof of Theorem \ref{GlobW2E}:}]
After dividing by $2^{10n}\|u\|_{L^{\infty}(B_1)}$ and adding a constant we may assume that $0 \leq u \leq 2^{-8n}$. We will show by induction that
\begin{equation}\label{Desired}
|B_1 \backslash A_{2^k}(u)| \leq (1-2^{-2}(2n\Lambda)^{1-n})^k|B_1|.
\end{equation}
The case $k = 0$ is obvious, so proceed to the inductive step.

Let $\rho_k = 2^{-4n-k/2}$. We claim that if a paraboloid $P$ of opening $-2^{k+1}$ is tangent from below to $\Gamma_u^{2^k}$ in $B_{1-4\rho_k}$, then when we slide it up until it touches $u$ the new
contact point is in $B_{1-\rho_k}$. Indeed, since $0 \leq \Gamma^{2^k}_u \leq u \leq 2^{-8n} = 2^k\rho_k^2$, it is easy to see that the vertex of $P$ is in $B_{1-3\rho_k}$. For the same reason, when we
slide $P$ up, the new contact point is a distance at most $\rho_k$ from the vertex.

It follows that if we define $F_k = B_{1-4\rho_k} \backslash A_{2^k}(u),$
and define $E_{k}$ as in the proof of Theorem \ref{LocalW2E}, then $E_{k} \subset B_{1-\rho_k} \subset\subset B_1$. Applying Lemma \ref{MeasureEstimate} we conclude that
\begin{equation}\label{MeasureBump}
2^{-1}(2n\Lambda)^{1-n}|F_k| \leq |A_{2^{k+1}}(u) \backslash A_{2^k}(u)|.
\end{equation}
There are two cases to consider:

\vspace{3mm}

{\bf Case 1:} $|F_k| \geq \frac{1}{2}|B_1 \backslash A_{2^k}(u)|$. Then (\ref{Desired}) follows from inequality (\ref{MeasureBump}) by arguing exactly as in the proof of Theorem \ref{LocalW2E}.

\vspace{3mm}

{\bf Case 2:} $|F_k| < \frac{1}{2}|B_1 \backslash A_{2^k}(u)|$. Then 
\begin{align*}
|B_1 \backslash A_{2^{k+1}}(u)| \leq |B_1 \backslash A_{2^k}(u)| &\leq 2 |B_1 \backslash B_{1-4\rho_k}| \\
&= 2|B_1|(1-(1-4\rho_k)^n) \\
&\leq 8n\rho_k|B_1|\\
&< 2^{-\frac{k+1}{2}}|B_1|,
\end{align*}
using the definition of $\rho_k$. By elementary calculus we have
$$\log(1-2^{-2}(2n\Lambda)^{1-n}) \geq -2^{-1}(2n\Lambda)^{1-n} \geq -2^{-1}\log(2),$$
completing the induction. The desired inequality for $t \geq 2$ with $\epsilon = 2^{-3}(2n\Lambda)^{1-n}$ follows in the same way as described in the proof of Theorem \ref{LocalW2E}.
\end{proof}

\section{Weak Harnack Inequality}\label{WeakHarnack}

In this section we prove Theorem \ref{LEps}. Our strategy is the same as for the global $W^{2,\,\epsilon}$ estimate.
The new difficulty is that we do not have $L^{\infty}$ control of $u$, so we require one extra ingredient.
This is provided by the following estimate, which we view as a bound on the size of $u$ at scale $\rho$:
\begin{lem}\label{Growth}
Assume that $M_{\Lambda}^-(D^2u) \leq 0$ in $B_4 \subset \mathbb{R}^n$, and $u \geq 0$. Then 
$$\inf_{B_{\rho}(x_0)}u \leq 2\, \rho^{-n\Lambda}u(0)$$
for all $x_0 \in B_{1}$ and $\rho \leq 1$.
\end{lem}
\begin{proof}
If $0 \in B_{\rho}(x_0)$ the conclusion is obvious, so assume not. For $x_0 \in B_1$ let $w := (|x-x_0|^{-n\Lambda}-2^{-n\Lambda})/(1-2^{-n\Lambda}).$
It is easy to compute $M_{\Lambda}^-(D^2w) > 0$ away from $x_0$, that $w < 0$ on $\partial B_4$, and that $w(0) > 1$. If $u > 2\,\rho^{-n\Lambda}u(0)\, (> u(0)w)$ on 
$\partial B_{\rho}(x_0)$, we contradict the maximum principle.
\end{proof}

\begin{proof}[{\bf Proof of Theorem \ref{LEps}}:]
After dividing by $2^{64n^2\Lambda}u(0)$ we may assume that $u(0) = 2^{-64n^2\Lambda}$. 
We first claim that $A_{2^k}(u) \cap B_{1/2} \subset \{u \leq 2^k\}$ for all $k \geq 0$. Indeed, if there exists $x_0 \in A_{2^k}(u) \cap B_{1/2}$ with $u(x_0) > 2^k$, then
$u \geq 2^{k-1}$ in a ball of radius $1$ containing $x_0$, hence in some ball $B_{1/4}(x_1)$ with $x_1 \in B_1$. However, by Lemma \ref{Growth} we have 
$\inf_{B_{1/4}(x_1)}u \leq 2^{1+2n\Lambda - 64n^2\Lambda} < 2^{-1}$, a contradiction.

Now, let $\rho_k := 2^{-30n}\,2^{-\frac{k+1}{2n\Lambda}}$ for $k \geq 0$. We claim that 
\begin{equation}\label{IteratedGrowth}
\inf_{B_{\rho_k}(x)}(2^{-k}u) \leq \rho_k^2 \text{  for all  } x \in B_1.
\end{equation}
This is a direct consequence of Lemma \ref{Growth}, which gives
$$\rho_k^{-2} \inf_{B_{\rho_k}(x)}(2^{-k}u) \leq 2^{2 + 60n - 34n^2\Lambda}\,2^{-(k+1)(\frac{1}{2} - \frac{1}{n\Lambda})} \leq 2^{2-8n} < 1.$$

We will show by induction that
\begin{equation}\label{DesiredWH}
|B_{1/2} \backslash A_{2^k}(u)| \leq (1-2^{-2}(2n\Lambda)^{1-n})^k |B_{1/2}|.
\end{equation}
The case $k = 0$ is obvious, so we proceed to the inductive step.

Assume that a paraboloid $P$ of opening $-2$ touches $\Gamma_{2^{-k} u}^1$ from below at $x_0 \in B_{1/2 - 5\rho_k}$, and 
slide it up until it touches $2^{-k}u$. We claim that the contact point is in $B_{1/2 - \rho_k}$. Indeed, note that $0 \leq \Gamma_{2^{-k}u}^1 \leq 2^{-k}u$. 
From (\ref{IteratedGrowth}) it is easy to see that the vertex $x_1$ of $P$ is in $B_1$.
If $|x_1 - x_0| \geq 2\rho_k$ then $2^{-k}u \geq 3\rho_k^2$ in $B_{\rho_k}(x_1)$, contradicting (\ref{IteratedGrowth}), so $x_1 \in B_{1/2 - 3\rho_k}$. Similar reasoning shows that
when we lift $P$ the new contact point must be a distance at most $2\rho_k$ from $x_1$, proving the claim.

It follows that for $F_k := B_{1/2 - 5\rho_k} \backslash A_{2^k}(u)$ and $E_{k}$ as in the proof of Theorem \ref{GlobW2E}, we have $E_{k} \subset B_{1/2 - \rho_k} \subset \subset B_{1/2}$.
We conclude using Lemma \ref{MeasureEstimate} that
\begin{equation}\label{MeasureBumpWH}
|(A_{2^{k+1}}(u) \backslash A_{2^k}(u)) \cap B_{1/2}| \geq 2^{-1}(2n\Lambda)^{1-n}|F_k|.
\end{equation}
(Actually, we use a small modification of Lemma \ref{MeasureEstimate} where we assume that $E \subset\subset \Omega' \subset \Omega$ and conclude that 
$|(A_{2a}(u) \backslash A_a(u)) \cap \Omega'| \geq 2^{-1}(2n\Lambda)^{1-n}|F|,$ which is easy to see from the proof). There are two cases to consider.

\vspace{3mm}

{\bf Case $1$:} $|F_k| \geq \frac{1}{2} |B_{1/2} \backslash A_{2^k}(u)|$. Then (\ref{DesiredWH}) follows from (\ref{MeasureBumpWH}) by arguing exactly as in the proof of Theorem \ref{LocalW2E}
and applying the inductive hypothesis.

\vspace{3mm}

{\bf Case $2$:} $|F_k| < \frac{1}{2} |B_{1/2} \backslash A_{2^k}(u)|$. Then we have
\begin{align*}
|B_{1/2} \backslash A_{2^{k+1}}(u)| \leq |B_{1/2} \backslash A_{2^k}(u)| &\leq 2|B_{1/2} \backslash B_{1/2-5\rho_k}| \\
&\leq 20n\rho_k |B_{1/2}| \\
&< 2^{-\frac{k+1}{2n\Lambda}}|B_{1/2}|.
\end{align*}
Elementary calculus gives $2^{-\frac{1}{2n\Lambda}} < 1-2^{-2}(2n\Lambda)^{1-n}$, completing the induction.

The inclusion $\{u > 2^k\} \cap B_{1/2} \subset B_{1/2} \backslash A_{2^k}(u)$ gives
$$|\{u > 2^k\} \cap B_{1/2}| \leq (1-2^{-2}(2n\Lambda)^{1-n})^k|B_{1/2}|,$$
and the desired inequality for $t \geq 2$ with $\epsilon = 2^{-3}(2n\Lambda)^{1-n}$ follows.
\end{proof}





\end{document}